\documentclass[twoside]{article}

\usepackage{amsmath,amssymb,amsthm}
\usepackage[accepted]{aistats2021}

\usepackage{bm}
\usepackage{graphicx,float,wrapfig}
\usepackage{subcaption}
\usepackage{wrapfig}
\usepackage{ifpdf}
\usepackage{listings}
\usepackage{diagbox}
\usepackage{algorithm}
\usepackage{algorithmicx}
\usepackage[noend]{algpseudocode}
\usepackage{enumitem}
\usepackage{multirow}
\usepackage{multicol}
\usepackage{color}
\usepackage{soul}
\usepackage{hhline}
\usepackage{thm-restate}
\usepackage{graphicx}
\usepackage[font=small,labelfont=bf]{caption}
\usepackage{hyperref}
\usepackage[round]{natbib}

\newcommand{\real}{\mathbb{R}}
\newcommand{\realp}{\Re}
\newcommand{\imag}{\Im}
\newcommand{\comp}{\mathbb{C}}

\newcommand{\bx}{\mathbf{x}}
\newcommand{\by}{\mathbf{y}}
\newcommand{\bA}{\mathbf{A}}
\newcommand{\bB}{\mathbf{B}}
\newcommand{\bC}{\mathbf{C}}
\newcommand{\bigo}{\mathcal{O}}
\newcommand{\set}{\mathcal{S}}
\newcommand{\bz}{\mathbf{z}}
\newcommand{\bb}{\mathbf{b}}

\newcommand{\zero}{\mathbf{0}}

\newcommand{\jacobian}{\mathbf{J}}
\newcommand{\iden}{\mathbf{I}}

\newtheorem{theorem}{Theorem}
\newtheorem{rem}{Remark}

\newtheorem{prop}{Proposition} 
\newtheorem{assumption}{Assumption} 
\newtheorem{examp}{Example}

\declaretheorem[name=Lemma]{lemma}
\newtheorem{definition}{Definition}

\DeclareMathOperator*{\argmin}{arg\,min}

\newenvironment{tightcenter}{%
  \setlength\topsep{0pt}
  \setlength\parskip{3pt}
  \begin{center}
}{%
  \end{center}
}

\definecolor{mygreen}{rgb}{0.1,0.9,0.1}
\definecolor{mydarkred}{rgb}{0.7,0.1,0}
\hypersetup{colorlinks=true,
    linkcolor=mygreen,
    citecolor=mygreen,
    filecolor=mygreen,
    urlcolor=mygreen}

\begin{document}

\twocolumn[

\aistatstitle{On the Suboptimality of Negative Momentum \\ for Minimax Optimization}

\aistatsauthor{ Guodong Zhang \And Yuanhao Wang }

\aistatsaddress{ University of Toronto, Vector Institute \And  Princeton University } ]

\begin{abstract}
Smooth game optimization has recently attracted great interest in machine learning as it generalizes the single-objective optimization paradigm. 
However, game dynamics is more complex due to the interaction between different players and is therefore fundamentally different from minimization, posing new challenges for algorithm design.
Notably, it has been shown that negative momentum is preferred due to its ability to reduce oscillation in game dynamics. Nevertheless, the convergence rate of negative momentum was only established in simple bilinear games. 
In this paper, we extend the analysis to smooth and strongly-convex strongly-concave minimax games by taking the variational inequality formulation. By connecting Polyak’s momentum with Chebyshev polynomials, we show that negative momentum accelerates convergence of game dynamics locally, though with a suboptimal rate. To the best of our knowledge, this is the \emph{first work} that provides an explicit convergence rate for negative momentum in this setting.
\end{abstract}

\section{Introduction}
Due to the increasing popularity of generative adversarial networks~\citep{goodfellow2014generative, radford2015unsupervised, arjovsky2017wasserstein}, adversarial training~\citep{madry2018towards} and primal-dual reinforcement learning~\citep{du2017stochastic, dai2018sbeed}, minimax optimization (or generally game optimization) has gained significant attention as it offers a flexible paradigm that goes beyond ordinary loss function minimization.
In particular, our problem of interest is the following minimax optimization problem:
\begin{equation}
    \min_{\bx \in \mathcal{X}} \max_{\by \in \mathcal{Y}} f(\bx, \by).
\end{equation}
We are usually interested in finding a \emph{Nash equilibrium}~\citep{von1944theory}: a set of parameters from which no player can (locally and unilaterally) improve its objective function.
Though the dynamics of gradient based methods are well understood for minimization problems, new issues and challenges appear in minimax games. For example, the na{\"\i}ve extension of gradient descent can fail to converge~\citep{letcher2019differentiable, mescheder2017numerics} or converge to undesirable stationary points~\citep{mazumdar2019finding, adolphs2019local, wang2019solving}.

Another important difference between minimax games and minimization problems is that negative momentum value is preferred for improving convergence~\citep{gidel2019negative}. To be specific, for the blinear case $f(\bx, \by) = \bx^\top \bA \by$, negative momentum with alternating updates converges to $\epsilon$-optimal solution with an iteration complexity of $\bigo(\kappa)$ where the condition number $\kappa$ is defined as $\kappa = \frac{\lambda_\mathrm{max}(\bA^\top \bA)}{\lambda_\mathrm{min}(\bA^\top \bA)}$, whereas Gradient Descent Ascent (GDA) fails to converge. Moreover, the rate of negative momentum matches the optimal rate of Extra-gradient (EG)~\citep{Korpelevich1976TheEM} and Optimistic Gradient Descent Ascent (OGDA)~\citep{daskalakis2018training, mertikopoulos2018optimistic}. 
A natural question to ask then is:
\begin{tightcenter}
    \emph{Does negative momentum improve on GDA \\ for other settings?}
\end{tightcenter}
In this paper, we extend the analysis of negative momentum\footnote{Throughout the paper, negative momentum represents gradient descent-ascent with negative momentum.} to the strongly-convex strongly-concave setting
and answer the above question in the affirmative. 
In particular, we observe that momentum methods~\citep{polyak1964some}, either positive or negative, can be connected to Chebyshev iteration~\citep{manteuffel1977tchebychev} in solving linear systems, which enables us to derive the optimal momentum parameter and asymptotic convergence rate.
With optimally tuned parameters, negative momentum achieves an acceleration \emph{locally} with an improved iteration complexity $\bigo(\kappa^{1.5})$ as opposed to the $\bigo(\kappa^2)$ complexity of Gradient Descent Ascent (GDA). Following on that, we further ask:
\begin{tightcenter}
    \emph{Is negative momentum optimal in the same setting? \\ Does it match the iteration complexity of EG and OGDA again?}
\end{tightcenter}
We answer these questions in the negative. Particularly, our analysis implies that the iteration complexity lower bound for negative momentum is $\Omega(\kappa^{1.5})$. Nevertheless, the optimal iteration complexity for this family of problems under first-order oracle is $\Omega(\kappa)$~\citep{ibrahim2019linear, zhang2019lower}, which can be achieved by EG and OGDA. Therefore, we \emph{for the first time} show that negative momentum alone is suboptimal for strongly-convex strongly-concave minimax games. To the best of our knowledge, this is the first work that provides an explicit convergence rate for negative momentum in this setting.

\textbf{Organization}. In Section~\ref{sec:preliminary}, we define our notation and formulate minimax optimization as a variational inequality problem. Under the variational inequality framework, we further write first-order methods as discrete dynamical systems and show that we can safely linearize the dynamics for proving local convergence rates (thus simplifying the problem to that of solving linear systems).
In Section~\ref{sec: poly-based-methods}, we discuss the connection between first-order methods and polynomial approximation and show that we can analyze the convergence of a first-order method through the sequence of polynomials it defines. In Section~\ref{sec:main-result}, we prove the local convergence rate of negative momentum for minimax games by connecting it with Chebyshev polynomials, showing that it has a suboptimal rate locally. Finally, in Section~\ref{sec:simulation}, we validate our claims in simulation.

\section{Preliminaries}\label{sec:preliminary}
\textbf{Notation}. In this paper, scalars are denoted by lower-case letters (e.g., $\lambda$), vectors by lower-case bold letters (e.g., $\bz$), matrices by upper-case bold letters (e.g., $\jacobian$) and operators by upper-case letters (e.g., $F$). 
The superscript $^\top$ represents the transpose of a vector or a matrix.
The spectrum of a square matrix $\bA$ is denoted by $\text{Sp}(\bA)$, and its eigenvalue by $\lambda$. We use $\realp$ and $\imag$ to denote the real part and imaginary part of a complex scalar respectively. 
We use $\real$ and $\comp$ to denote the set of real numbers and complex numbers, respectively.
We use $\rho(\bA) = \lim_{n \rightarrow \infty} \|\bA^n\|^{1/n}$ to denote the spectral radius of matrix $\bA$. $\bigo$, $\Omega$ and $\Theta$ are standard asymptotic notations. We use $\Pi_t$ to denote the set of real polynomials with degree no more than $t$.

\begin{table*}[t]
    \centering
    \caption{First-order algorithms for smooth and strongly-monotone games.}
    \begin{tabular}{ |c|c|c|c| } 
     \hline
     Method & Parameter Choice & Complexity & Reference \\ \hline
     GDA & $\alpha = 0$, $\beta = 0$ & $\bigo(\kappa^2)$ & \small{\citet{ryu2016primer, azizian2020tight}} \\ \hline
     OGDA & $\alpha = 1$, $\beta = 0$ & $\bigo(\kappa)$ & \small{\citet{gidel2018variational, mokhtari2020unified}} \\ \hline
     NM & $\alpha = 0$, $\beta < 0$  & $\Theta(\kappa^{1.5})$ & \textbf{This paper} (Theorem~\ref{thm:main-result})\\ \hline
    \end{tabular}
    \label{tab:first-order}
\end{table*}
\subsection{Variational Inequality Formulation of Minimax Optimization}
We begin by presenting the basic variational inequality framework that we will consider throughout the paper. To that end, let $\mathcal{Z}$ be a nonempty convex subset of $\real^d$, and let $F: \real^d \rightarrow \real^d$ be a continuous mapping on $\real^d$. In its most general form, the variational inequality (VI) problem~\citep{harker1990finite} associated to $F$ and $\mathcal{Z}$ can be stated as:
\begin{equation}\label{eq:variational-inequality}
    \text{find } \bz^* \in \mathcal{Z} \;\; \text{s.t.} \;\; F(\bz^*)^\top(\bz - \bz^*) \geq 0 \;\; \forall \bz \in \mathcal{Z}.
\end{equation}
In the case of $\mathcal{Z} = \real^d$, the problem is reduced to finding $\bz^*$ such that $F(\bz^*) = 0$. To provide some intuition about the variational inequality problem, we discuss two important examples below:
\begin{examp}[Minimization]
Suppose that $F = \nabla_\bz f$ for a smooth function $f$ on $\real^d$, then the varitional inequality problem is essentially finding the critical points of $f$. In the case where $f$ is convex, any solution of \eqref{eq:variational-inequality} would be a global minimum.%
\end{examp}
\begin{examp}[Minimax Optimization]\label{examp:minimax}
Consider the convex-concave minimax optimization (or saddle-point optimization) problem, where the objective is to solve the following problem
\begin{equation}
    \min_\bx \max_\by f(\bx, \by), \; \text{where } f \text{ is a smooth function.}
\end{equation}
One can show that it is a special case of \eqref{eq:variational-inequality} with $F(\bz) = [\nabla_\bx f(\bx, \by)^\top, -\nabla_\by f(\bx, \by)^\top]^\top$.
\end{examp}
Notably, the vector field $F$ in Example~\ref{examp:minimax} is not necessarily conservative, i.e., it
might not be the gradient of any function. In addition, since $f$ in minimax problem happens to be convex-concave, any solution $\bz^* = [{\bx^*}^\top, {\by^*}^\top]^\top$ of \eqref{eq:variational-inequality} is a global \emph{Nash Equilibrium}~\citep{von1944theory},~i.e.,
\begin{equation*}
    f(\bx^*, \by) \leq f(\bx^*, \by^*) \leq f(\bx, \by^*) \quad \forall \bx \text{ and } \by \in \real^d.
\end{equation*}
In this work, we are particularly interested in the case of $f$ being a strongly-convex-strongly-concave and smooth function, which essentially assumes that the vector field $F$ is strongly-monotone and Lipschitz (see~\citet[Lemma~2.6]{fallah2020optimal}). Here we state our assumptions formally.
\begin{assumption}[Strongly Monotone]\label{ass:monotonicity}
    The vector field $F$ is $\mu$-strongly-monotone:
    \begin{equation}
        (F(\bz_1) - F(\bz_2))^\top (\bz_1 - \bz_2) \geq \mu \|\bz_1 - \bz_2 \|_2^2, \; \forall \bz_1, \bz_2 \in \real^d.
    \end{equation}
\end{assumption}
\begin{assumption}[Lipschitz]\label{ass:lipschitz}
    The vector field $F$ is L-Lipschitz if the following holds:
    \begin{equation}
        \|F(\bz_1) - F(\bz_2) \|_2 \leq L \|\bz_1 - \bz_2 \|_2, \; \forall \bz_1, \bz_2 \in \real^d.
    \end{equation}
\end{assumption}
In the context of variational inequalites, Lipschitzness and (strong) monotonicity are fairly standard and have been used in many classical works~\citep{tseng1995linear, chen1997convergence, nesterov2007dual, nemirovski2004prox}. With these two assumptions in hand, we define the condition number $\kappa \triangleq L / \mu$, which measures the hardness of the problem. In the following, we turn to suitable optimization techniques for the variational inequality problem.

\subsection{First-order methods for Minimax Optimization}
The dynamics of optimization algorithms are often described by a vector field, $F$, and local convergence behavior can be understood in terms of the spectrum of its Jacobian. In minimization, the Jacobian coincides with the Hessian of the loss with all eigenvalues real. 
In minimax optimization, the Jacobian is nonsymmetric and can have complex eigenvalues, making it harder to analyze.

In the case of strongly-convex strongly-concave minimax games, finding the Nash equilibrium is equivalent to solving the fixed point equation $F(\bz^*) = \zero$.
Here, we mainly focus on first-order methods~\citep{nesterov1983method} to find the stationary point $\bz^*$:
\begin{definition}[First-order methods]\label{def:first-order-methods}
A first-order method generates
\begin{equation}\label{eq:first-order}
    \bz_t \in \bz_0 + \textbf{Span}\{F(\bz_0), ..., F(\bz_{t-1}) \}.
\end{equation}
\end{definition}
This wide class includes most gradient-based optimization methods we are interested in, such as GDA, OGDA and momentum method. All three methods are special case of the following update:
\begin{equation}
    \bz_{t+1} = (1 + \beta) \bz_{t} - \beta \bz_{t-1} - \eta F((1 + \alpha)\bz_t - \alpha \bz_{t-1}),
\end{equation}
where $\beta$ is the momentum parameter, $\alpha$ the extrapolation parameter and $\eta$ the step size. With proper choices of parameters, we can recover GDA, OGDA and negative momentum (see Table~\ref{tab:first-order}). For instance, the update rule of negative momentum is given by
\begin{equation}
    \label{eq:neg-momentum-def}
    \bz_{t+1} = (1 + \beta) \bz_{t} - \beta \bz_{t-1}-\eta F(\bz_t).
\end{equation}
For smooth and strongly-monotone games, \citet[Corollary 1]{azizian2020accelerating} showed a lower bound on convergence rate for \emph{any} algorithm of the form~\eqref{eq:first-order}:
\begin{equation}\label{eq:lower-bound}
\begin{aligned}
    & \|\bz_k - \bz^* \|_2 \geq \rho_{\text{opt}}^k \|\bz_0 - \bz^*\|_2 \\ & \text{with} \quad \rho_{\text{opt}} = 1 - \frac{2\mu}{\mu + L}.
\end{aligned}
\end{equation}

\subsection{Dynamical System Viewpoint and Local Convergence}
With a first-order algorithm defined, we study local convergence rates from the viewpoint of dynamical system.
It is well-known that gradient-based methods can reliably find local stable fixed points (i.e., local minima) in single-objective optimization.
Here, we generalize the concept of stability to games by taking game dynamics as a discrete dynamical system.
An iteration of the form $\bz_{t+1}=G(\bz_t)$ can be viewed as a discrete dynamical system. If $G(\bz^*)=\bz^*$, then $\bz^*$ is called a fixed point. We study the stability of fixed points as a proxy to local convergence of game dynamics.
\begin{definition}\label{def:stable-fixed-points}
Let $\jacobian_G$ denote the Jacobian of $G$ at a fixed point $\bz^*$. If it has spectral radius $\rho(\jacobian_G)\le 1$, then we call $\bz^*$ a stable fixed point. If $\rho(\jacobian_G)<1$, then we call $\bz^*$ a strictly stable fixed point.
\end{definition}
It has been shown that strict stability implies local convergence (see~\citet{galor2007discrete}). In other words, if $\bz^*$ is a strictly stable fixed point, there exists a neighborhood $U$ of $\bz^*$ such that when initialized in $U$, the iteration steps always converge to $\bz^*$.
\begin{rem}
    Because we focus on local convergence rates, we can safely take the Jacobian $\jacobian$ as constant locally, which essentially linearizes the vector field $F(\bz) = \bA\bz + \bb, \bA = \jacobian_F(\bz^*)$. Therefore, locally solving the minimax game becomes arguably as easy as solving the linear system $\bA\bz + \bb = \zero$.
\end{rem}

\subsection{Chebyshev Polynomials}\label{subsec:chebyshev}
The Chebyshev polynomials were discovered a century ago by the mathematician Chebyshev. 
Since then, they have found many uses in numerical analysis~\citep{fox1968chebyshev, mason2002chebyshev}. The Chebyshev polynomials can be defined recursively as%
\begin{equation}\label{eq:cheby-recursion}
\begin{aligned}
    T_0(z) &= 1, \; T_1(z) = z, \\
    T_{n+1}(z) &= 2z T_n(z) - T_{n-1}(z).
\end{aligned}
\end{equation}
They may also be written as
\begin{equation}
    T_n(z) = 
    \begin{cases}
    \cos (n \arccos (z)) & \quad \text{if } -1 \leq z \leq 1 \\
    \cosh (n \cosh^{-1} (z)) & \quad \text{otherwise}
    \end{cases}.
\end{equation}
For the case of $z \notin [-1, 1]$, we have the Chebyshev polynomials $T_n(z) = \cosh (n \cosh^{-1} (z))$. Consider the map $\eta = \cosh(\sigma)$, let $\sigma = x + y i$ and $\eta = u + v i$. Then $\cosh(\sigma) = \cosh(x + y i) = u + v i = \eta$. By the property of $\cosh$, we have
\begin{equation}
    \cosh(x + y i) = \cosh(x) \cos(y) + \sinh(x) \sin(y) i.
\end{equation}
If we fix $x = \mathrm{const}$ (with varying $y$), then we have
$\frac{u^2}{\cosh(x)^2} + \frac{v^2}{\sinh(x)^2} = 1$.
That is, $\cosh$ maps the vertical line $x = \mathrm{const}$ to an ellipse with semi-major axis $|\cosh(x)|$, semi-minor axis $|\sinh(x)|$ and foci at $+1$ and $-1$. This map has the period $2\pi i$.

\section{Polynomial-based Iterative Methods}\label{sec: poly-based-methods}
In the background section, we showed that solving minimax games locally boils down to solving a linear system. Here, we leverage the well-established theory of polynomial approximation for efficiently solving linear systems. The next lemma shows that when the vector field $F$ is linear, first-order algorithms defined in \eqref{eq:first-order} can be written as polynomials.
\begin{lemma}[\citet{fischer2011polynomial}] \label{lem:first-order}
If the vector field is of the form of $F(\bz) = \bA\bz + \bb$, then $\bz_t$ generated by first-order methods can be written as
\begin{equation}\label{lem:poly}
    \bz_t - \bz^* = p_t(\bA)(\bz_0 - \bz^*),
\end{equation}
where $\bz^*$ satisfies $\bA\bz^* + \bb = \zero$ and $p_t \in \Pi_t$ is a polynomial with degree at most $t$ that satisfies $p_t(0) = 1$.
\end{lemma}
To gain better intuition about the above Lemma, we provide an example of gradient descent below.
\begin{examp}[Gradient Descent]
    For gradient descent with constant learning rate $\eta$, the corresponding polynomials are given by
    \begin{equation}
        p_t(\bA) = (\iden - \eta \bA)^t.
    \end{equation}
\end{examp}
Hence, the convergence of a first-order method can be analyzed through the sequence of polynomials $p_t$ it generates. Specifically, we can bound the error of $\|\bz_t - \bz^*\|_2$ in a  as follows:
\begin{equation}\label{eq: upper-bound}
    \|\bz_t - \bz^*\|_2 \leq \|p_t(\bA)\|_2 \| \bz_0 - \bz^*\|_2,
\end{equation}
where $\|p_t(\bA)\|_2^{1/t} = (\max_{\lambda \in \text{Sp}(\bA)} |p_t(\lambda)|)^{1/t}$ as $t \rightarrow \infty$.
Importantly, the error depends on two factors: the polynomial (algorithm) $p_t$ and the matrix (problem) $\bA$. 
In practice, we are interested in the performance of a algorithm on a broad class of problem, therefore we instead consider a set $\set_K$ of matrices $\bA$:
\begin{equation}
    \set_K := \{\bA \in \real^{d \times d}: \text{Sp}(\bA) \in K \in \comp \}
\end{equation}
Clearly, an obvious choice for the residual polynomial $p_t$ is the one which minimizes the upper bound in~\eqref{eq: upper-bound}. This optimal polynomial $\mathcal{P}_t(\lambda, K)$ is the solution of the following Chebyshev approximation problem
\begin{equation*}
    \max_{\lambda \in K} |\mathcal{P}_t(\lambda; K)| = \min \left\{\max_{\lambda \in K} |p(\lambda)|: \; p \in \Pi_t, p(0) = 1 \right\}.
\end{equation*}
To measure the performance, we define the \emph{asymptotic convergence factor}~\citep{eiermann1983construction} with the following form:
\begin{equation}
    \rho(K) = \lim_{t \rightarrow \infty} \left( \max_{\lambda \in K} |\mathcal{P}_t(\lambda; K)|  \right)^{1/t} .
\end{equation}
It was shown that the asymptotic convergence factor also serves as a \textbf{lower bound} in the worse-case \citep{nevanlinna1993convergence} depending on the set $K$\footnote{It is the fastest possible asymptotic convergence rate a first-order method can achieve for all linear systems with spectrum in the set.}. 
\begin{prop}[\cite{nevanlinna1993convergence}]
Let $K$ be a subset of $\comp$ symmetric w.r.t the real axis, that does not contain the origin. Then, any oblivious first-order method (whose coefficients only depend on $K$, see~\citet{arjevani2016iteration}) satisfies the following
\begin{equation*}
    \forall t > 0, \exists \bz_0, \exists \bA: \|\bz_t - \bz^* \|_2 \geq \rho(K) \|\bz_0 - \bz^* \|_2.
\end{equation*}
\end{prop}
Interestingly, if the set $K$ is simple enough, we can compute the asymptotic convergence factor and the optimal polynomial. In particular, when $K$ is a complex ellipse in the complex plane which does not contain the origin in its interior, the following result is known in the literature~\citep{clayton1963further,wrigley1963accelerating,manteuffel1977tchebychev, azizian2020accelerating}.%
\begin{theorem}\label{thm: negative-momentum}
If the set of $K$ is a complex ellipse with the following form:
\begin{equation}\label{eq:ellipse}
\begin{aligned}
    K =& \left\{\lambda \in \comp: E_{a, b, d}(\lambda) \triangleq \frac{(\realp\lambda - d)^2}{a^2} + \frac{(\imag\lambda)^2}{b^2} \leq 1 \right\}, \\ & \text{with } d > a > 0, b > 0.
\end{aligned}
\end{equation}
then one can show its asymptotically optimal polynomial is a rescaled and translated Chebyshev polynomial:
\begin{equation}\label{eq: optimal-poly}
    \mathcal{P}_t(\lambda; K) = \frac{T_t\left(\frac{d - \lambda}{c} \right)}{T_t\left(\frac{d}{c} \right)}, \; c^2 = a^2 - b^2.
\end{equation}
Moreover, the asymptotic convergence factor (i.e., the maximum modulus) is achieved on the boundary:
\begin{equation}
    \rho(K) = 
    \begin{cases}
    a/d       & \quad \text{if } a = b \\
    \frac{d - \sqrt{d^2 + b^2 - a^2}}{a - b}  & \quad \text{otherwise}
  \end{cases}.
\end{equation}
\end{theorem}
\begin{rem}
    We note that $c$ can be either pure real or pure imaginary. In either case, $c^2$ is real and throughout the paper $c$ only appears as $c^2$.
\end{rem}
Notably, the first-order method corresponding to the optimal polynomial \eqref{eq: optimal-poly} is Polyak momentum.
\begin{restatable}{cor}{cornm}\label{cor:nm}
    The optimal first-order methods for $K$ in the form of \eqref{eq:ellipse} iterates as follows:
    \begin{equation}
        \bz_{t+1} = \bz_t - \eta_t F(\bz_t) + \beta_t (\bz_t - \bz_{t-1}),
    \end{equation}
    where $\eta_t$ and $\beta_t$ are not constant over time. However, by choosing constants $\eta = 2\frac{d - \sqrt{d^2 - c^2}}{c^2}$ and $\beta = d \eta - 1$,
    we can obtain the same asymptotic rate.
\end{restatable}
\begin{rem}\label{rem:optimal-mom}
The optimal momentum parameter $\beta$ is
\begin{equation}\label{eq:optim-beta}
    \beta = 2\tfrac{d^2 - d\sqrt{d^2 - c^2}}{c^2} - 1 = 2\tfrac{d^2 - \frac{1}{2}c^2 - \sqrt{d^2 (d^2 - c^2)}}{c^2}.
\end{equation}
In particular, we note that the numerator of \eqref{eq:optim-beta} is always non-negative. Therefore, we conclude that the optimal momentum parameter is negative when $c^2 = a^2 - b^2 < 0$ and positive when $c^2 > 0$.
\end{rem}
We note that the eigenvalues of the Jacobian for the minimization problem lie on the real axis, which is a special case of a complex ellipse with $b = 0$. In that case, it is known that Polyak momentum has an optimal worst-case convergence rate over the class of first order methods~\citep{polyak1987introduction}. In the special case of $K$ being a disc, we have the optimal algorithm being gradient descent~\citep{eiermann1985study}.
\begin{restatable}{cor}{corgda}
\label{cor:gda}
    For the case of $a = b$, i.e., K is a disc in the complex plane, the optimal polynomial is 
    \begin{equation}
        \mathcal{P}_t(\lambda; K) = (1 - \lambda / d)^t,
    \end{equation}
    and the optimal algorithm is gradient descent.
\end{restatable}
\begin{restatable}{cor}{cormodulus}
\label{cor:max-modulus}
    The modulus of $\mathcal{P}_t(\lambda; K)$ is constant on the boundary of \eqref{eq:ellipse} for large $t$.
\end{restatable}

\section{Suboptimality of Negative Momentum}\label{sec:main-result}
In the previous section, we have shown that Polyak momentum with properly chosen parameters is asymptotically optimal for linear systems with the spectrum enclosed in the region of complex ellipse. In this section, we shift our attention back to minimax games. In particular, we analyze minimax optimization with the framework of variational inequality.

\begin{figure}[t]
	\centering
    \includegraphics[width=0.98\columnwidth]{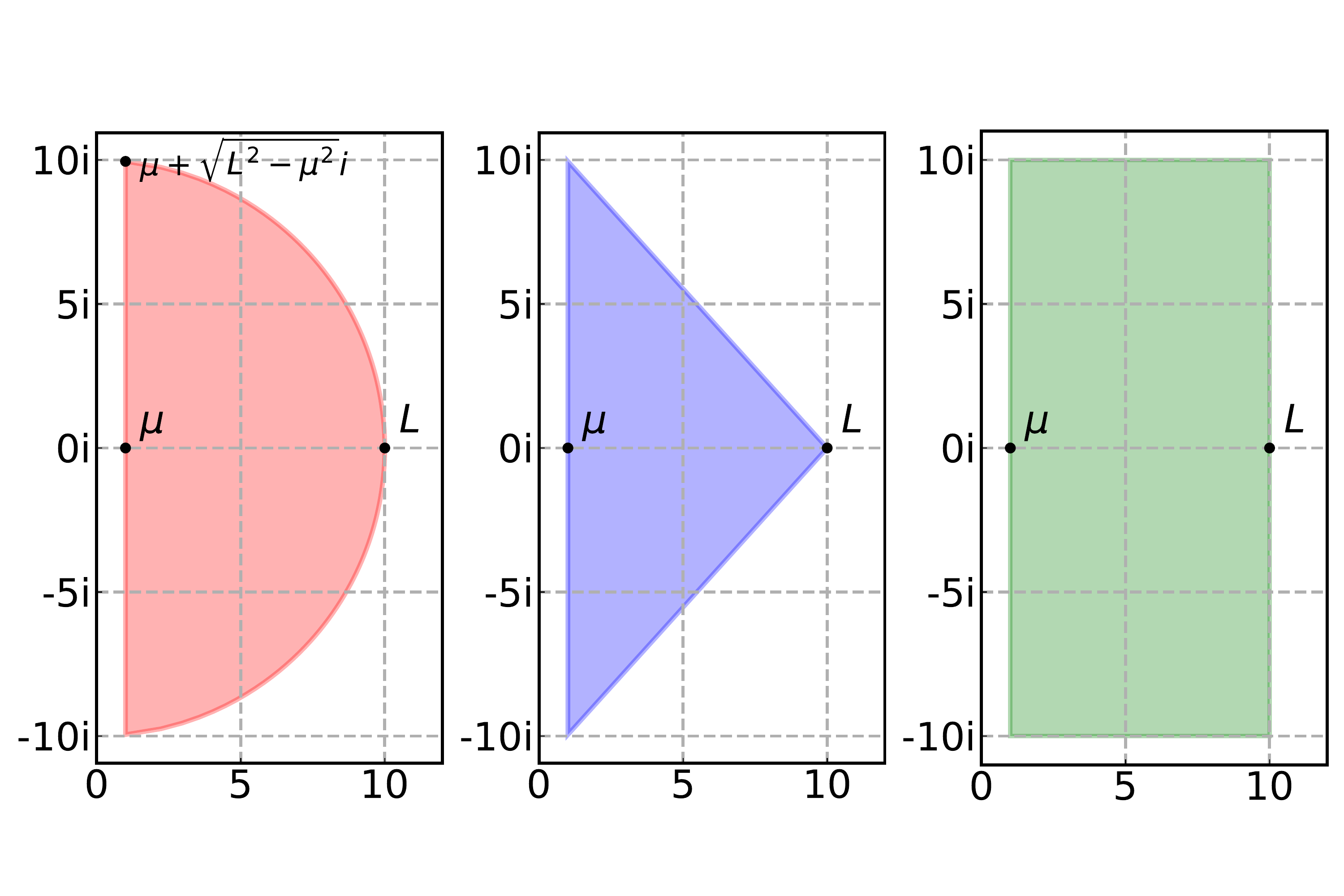}
    \vspace{-0.6cm}
    \caption{\textbf{Left:} The red region corresponds to $\hat{K}$, the set of strongly monotone problems. \textbf{Middle:} The blue triangle corresponds to $\hat{K}_1$, which is enclosed within $\hat{K}$. \textbf{Right:} The green rectangle is $\hat{K}_2$, which includes $\hat{K}$ in its interior. For visualization, we set $\mu = 1$ and $L = 10$.}
	\label{fig:region}
\end{figure}
Obviously, under Assumptions~\ref{ass:monotonicity} and~\ref{ass:lipschitz}, the eigenvalue of $\jacobian_F(\bz^*)$ will not tightly fall within a complex ellipse. It can be shown that it instead lies within the following set~\citep{azizian2020accelerating}:
\begin{equation}\label{eq: minimax-region}
    \hat{K} = \left\{ \lambda \in \comp: |\lambda| < L, \realp \lambda > \mu > 0 \right\}.
\end{equation}
This set is the intersection between a circle and a halfplane (see Figure~\ref{fig:region}). 

Recall that our goal is to search for the best achievable convergence rate of negative momentum (or generally Polyak momentum) for linear systems with spectrum enclosed within $\hat{K}$.
By linearizing the vector field locally $F(\bz) = \bA\bz + \bb = \bA(\bz-\bz^*)$ and expanding the state space to $[\bz_{t+1}^\top, \bz_t^\top]^\top$, we can write \eqref{eq:neg-momentum-def} in matrix form%
\begin{equation}\label{eq:aug-nm}
    \begin{bmatrix}
        \bz_{t + 1}-\bz^* \\ \bz_t-\bz^*
    \end{bmatrix} = \bar{\jacobian} \begin{bmatrix}
        \bz_{t}-\bz^* \\ \bz_{t-1}-\bz^*
    \end{bmatrix},
\end{equation}
where the matrix $\bar{\jacobian}$ has the following form:
\begin{equation}
    \bar{\jacobian} = \begin{bmatrix}  
    (1 + \beta) \iden - \eta \bA & - \beta \iden \\
    \iden & \zero
    \end{bmatrix}.
\end{equation}
Thus, finding the asymptotic convergence rate boils down to the following min-max problem
\begin{equation}\label{eq:minimax-nm}
    \hat{\rho}(\hat{K}) \triangleq \min_{\eta, \beta} \max_{\lambda \in \hat{K}} \rho \left( \begin{bmatrix}  
    1 + \beta - \eta \lambda & - \beta \\
    1 & 0
    \end{bmatrix} \right).
\end{equation}
Essentially, we would like to find the optimal step size $\eta$ and momentum parameter $\beta$ that minimize the spectral radius which determines the asymptotic convergence rate. However, due to the fact that the spectrum is in the complex plane and involves complex eigenvalues, bounding the spectral radius directly becomes challenging. Nevertheless, by Theorem~\ref{thm: negative-momentum} and Corollary~\ref{cor:nm}, we have the following equivalence:
\begin{restatable}[Asymptotic Equivalence between Polyak momentum and Chebyshev Iteration]{lemma}{lemreduction}
    For any $K \in \comp$ that is symmetric w.r.t the real axis and does not contain the origin,
    if Polyak momentum with parameters $\eta, \beta$ converges with rate $\rho < 1$, then there exists a rescaled and translated Chebyshev polynomial parameterized by $d, c^2 \in \real$ converging with the same asymptotic rate, and vice versa.
\end{restatable}
Hence the min-max problem \eqref{eq:minimax-nm} is equivalent to:
\begin{equation}\label{eq:min-max-prob}
\begin{aligned}
    &\hat{\rho}(\hat{K}) = \min_{d, c^2 \in \real} \max_{\lambda \in \hat{K}} r(\lambda; d, c^2), \\ &\text{where} \; r(\lambda; d, c^2) \triangleq \lim_{t\rightarrow \infty} \left|\frac{T_t(\frac{d - \lambda}{c})}{T_t(\frac{d}{c})} \right|^{1/t}.
\end{aligned}
\end{equation}
We term $r(\lambda; d, c^2)$ the \textbf{convergence factor} of Chebyshev polynomial $T_t(\frac{d - \lambda}{c})/T_t(\frac{d}{c})$ at the point $\lambda$.
The reason why we can do such a reduction is that momentum method is equivalent to the rescaled and translated Chebyshev polynomial \eqref{eq: optimal-poly} asymptotically, and different parameters $\eta, \beta$ exactly corresponds to different choices of $d, c^2$ in \eqref{eq: optimal-poly}. 

However, the equivalent min-max problem \eqref{eq:min-max-prob} is not easy to solve directly and some reductions have to be done. Our very first step is to use the the sandwich technique, which is inspired by~\citet{azizian2020accelerating}. Let $\hat{K}_1$ and $\hat{K}_2$ be the two regions tightly lower bounding and upper bounding $\hat{K}$ (see Figure~\ref{fig:region}).
\begin{equation}
\begin{aligned}
    \hat{K}_1  = \{&\lambda \in \comp: \realp \lambda \geq \mu, \tfrac{1}{L}\realp \lambda + \tfrac{L -\mu}{L\sqrt{L^2 - \mu}}\imag \lambda \leq 1, \\ & \tfrac{1}{L}\realp \lambda - \tfrac{L -\mu}{L\sqrt{L^2 - \mu}}\imag \lambda \leq 1 \}; \\
    \hat{K}_2 = \{&\lambda \in \comp: \mu \leq \realp \lambda \leq L, \\ &-\sqrt{L^2 - \mu^2} \leq \imag \lambda \leq \sqrt{L^2 - \mu^2} \}.
\end{aligned}
\end{equation}
One can see that both $\hat{K}_1$ and $\hat{K}_2$ are convex polygons and particularly $\hat{K}_1 \subset \hat{K} \subset \hat{K}_2$. Therefore, we have
\begin{equation}
    \hat{\rho}(\hat{K}_1)\leq \hat{\rho}(\hat{K}) \leq \hat{\rho}(\hat{K}_2).
\end{equation}
Now, the main challenge is to compute $\hat{\rho}(\hat{K}_1)$ and $\hat{\rho}(\hat{K}_2)$. Ideally, we would hope that they are close to each other and thus we can bound $\hat{\rho}(\hat{K})$ tightly. 
Given that $\hat{K}_1$ and $\hat{K}_2$ are convex polygons, we have the following results:
\begin{lemma}[{\citet[Lemma~3.2]{manteuffel1977tchebychev}}]\label{lem:lem3.2}
Defining $H_1$ and $H_2$ to be the sets of vertices of $\hat{K}_1$ and $\hat{K}_2$, we have
\begin{equation}\label{eq:convex-poly-reduction}
\begin{aligned}
    \hat{\rho}(\hat{K}_1) = \min_{d, c^2} \max_{\lambda \in H_1} r(\lambda; d, c^2), \\
    \hat{\rho}(\hat{K}_2) = \min_{d, c^2} \max_{\lambda \in H_2} r(\lambda; d, c^2).
\end{aligned}
\end{equation}
Both $H_1$ and $H_2$ are symmetric w.r.t the real axis, we can therefore reduce them to $H_1 = \{L, \mu + \sqrt{L^2 - \mu^2}i \}$ and $H_2 = \{L + \sqrt{L^2 - \mu^2}i, \mu + \sqrt{L^2 - \mu^2}i  \}$.
\end{lemma}
Essentially, Lemma~\ref{lem:lem3.2} says that, for optimal $d$ and $c^2$, the largest convergence factor occurs on the \emph{hull} of $\hat{K}_1$ and $\hat{K}_2$, i.e., the set including all the vertices. Therefore, we do not need to maximize over all elements of $\hat{K}_1$ and $\hat{K}_2$, which makes the problem much simpler.
Next, we apply the powerful Alternative theorem in functional analysis~\citep{bartle1964elements} to further simplify the min-max problem.
\begin{restatable}{lemma}{lemalter}\label{lem:alternative-thm}
    For optimal parameters $d_i^*, {c_i^2}^*$ in min-max problem \eqref{eq:convex-poly-reduction}, all points in $H_i$ has the same rate
    \begin{equation*}
    \begin{aligned}
        r(L; d_1^*, {c_1^2}^*) &= r(\mu + \sqrt{L^2 - \mu^2}i; d_1^*, {c_1^2}^*); \\
        r(L + \sqrt{L^2 - \mu^2}i; d_2^*, {c_2^2}^*) &= r(\mu + \sqrt{L^2 - \mu^2}i; d_2^*, {c_2^2}^*).\\
    \end{aligned}
    \end{equation*}
\end{restatable}
Intuitively, Lemma~\ref{lem:alternative-thm} suggests that vertices of $\hat{K}_i$ have the same convergence factor. 
As a consequence, one can show that all vertices of $\hat{K}_i$ are on the boundary of the same complex ellipse.
\begin{lemma}[{\citet{manteuffel1977tchebychev}}]\label{lem:convergence-factor}
    Let $\mathcal{E}(d, c^2)$ be the family of complex ellipse in the complex plane centered at $d$ with foci at $d - c$ and $d + c$. Further let $E(d, c^2) \in \mathcal{E}(d, c^2)$ be a member of this family that not include the origin in its interior. Then for any two points $\lambda_i \in E_i(d, c^2)$ and $\lambda_j \in E_j(d, c^2)$, we have
    \begin{equation*}
    \begin{aligned}
        r(\lambda_i; d, c^2) = r(\lambda_j; d, c^2)  &\iff E_i(d, c^2) = E_j(d, c^2) \\
        r(\lambda_i; d, c^2) < r(\lambda_j; d, c^2)  &\iff E_i(d, c^2) \subset E_j(d, c^2) \\
    \end{aligned}
    \end{equation*}
\end{lemma}
To understand this Lemma, one shall realize that the convergence factor can be further written as
\begin{equation*}
    r(\lambda; d, c^2) = e^{\cosh^{-1}(\tfrac{d - \lambda}{c}) - \cosh^{-1}(\tfrac{d}{c})} \propto e^{\cosh^{-1}(\tfrac{d - \lambda}{c})}
\end{equation*}
In particular, the transformation $\lambda \mapsto \tfrac{d - \lambda}{c}$ maps the points in $E(d, c^2)$ to $E(0, 1)$. By the property of $\cosh^{-1}$ (see Section~\ref{subsec:chebyshev}), $\cosh^{-1}(\tfrac{d - \lambda}{c})$ maps $E(d, c^2)$ to a vertical line $x = \cosh^{-1}(a)$ where $a$ is the semi-major axis of the specific $E(0, 1)$. So to compare the convergence factors of two points $\lambda_i$, $\lambda_j$, we only need to compare the semi-major axis of $E_i(d, c^2)$ and $E_j(d, c^2)$.

Finally, we are ready to present our main result.
\begin{theorem}[Suboptimality of Negative Momentum]\label{thm:main-result}
    Under Assumptions~\ref{ass:monotonicity} and~\ref{ass:lipschitz}, we have the optimal momentum parameter $\beta$ to be negative and
    \begin{equation*}
        \hat{\rho}(\hat{K}_1) = 1 - \Theta(\kappa^{-1.5}), \quad
        \hat{\rho}(\hat{K}_2) = 1 - \Theta(\kappa^{-1.5}).
    \end{equation*}
    By the sandwich argument, we therefore get $\hat{\rho}(\hat{K}) = 1 - \Theta(\kappa^{-1.5})$. Assuming the vector field $F$ is continuously differentiable, for $\bz_0$ close to $\bz^*$, negative momentum can converge to $\bz^*$ asymptotically with the rate $1 - \Theta(\kappa^{-1.5})$.
\end{theorem}
\begin{proof}[Proof sketch]
Here we give a short proof sketch with detailed proof deferred to the supplement.
Let's first prove the result for $\hat{\rho}(\hat{K}_1)$. By Lemma~\ref{lem:alternative-thm}, we have
\begin{equation*}
    r(L; d_1^*, {c_1^2}^*) = r(\mu+\sqrt{L^2 - \mu^2}i; d_1^*, {c_1^2}^*),
\end{equation*}
which implies that both $L$ and $\mu + \sqrt{L^2 - \mu^2}i$ are on the boundary of the same complex ellipse with the center $d_1^*$ and foci at $d_1^* - c_1^*$ and $d_1^* + c_1^*$ according to Lemma~\ref{lem:convergence-factor}. %
Then by Theorem~\ref{thm: negative-momentum} and Corollary~\ref{cor:max-modulus}, we can reduce the computation of $\hat{\rho}(\hat{K}_1)$ to the following constrained problem:
\begin{equation}\label{eq:constrainted-prob-main}
\begin{aligned}
    &\hat{\rho}(\hat{K}_1) := \min_{a, b, d} \tfrac{d - \sqrt{d^2 + b^2 - a^2}}{a - b}, \\ &\text{s.t. } E_{a,b,d}(L) = E_{a, b, d}(\mu + \sqrt{L^2 - \mu^2}i) = 1
\end{aligned}
\end{equation}
which involves three free variables and two constraints. The two constraint equations imply
\begin{equation*}
    b^2 = \frac{(L + \mu)(L - d)^2}{L + \mu - 2d} > (L - d)^2 = a^2.
\end{equation*}
Therefore, the optimal momentum $\beta$ for $\hat{K}_1$ is negative. For $\hat{K}_2$, we follow the same procedure and have
\begin{equation*}
    b^2 = \frac{(L^2 - \mu^2) a^2}{a^2 - (\frac{L - \mu}{2})^2}, \; a \in \left[\frac{L - \mu}{2}, \frac{L + \mu}{2}\right].
\end{equation*}
In the case of $L^2 > \mu^2 + \mu L$, we have $b^2 > a^2$ and therefore the optimal momentum is also negative since $c^2 = a^2 - b^2 < 0$ (see Remark~\ref{rem:optimal-mom}). Hence, we conclude that the optimal momentum for $\hat{K}$ is negative.

Next, we bound $\hat{\rho}(\hat{K}_1)$ and $\hat{\rho}(\hat{K}_2)$ so as to estimate $\hat{\rho}(\hat{K})$. Towards this end, one can further simplify the problem \eqref{eq:constrainted-prob-main} to a single variable minimization task:
\begin{equation}\label{eq:prob-tri-main}
    \min_{d \in [\frac{L}{2}, \frac{\mu+L}{2}]} \frac{d - \sqrt{\frac{2d(L - d)^2}{L + \mu - 2d} + d^2}}{(L - d)(1 - \sqrt{\frac{L + \mu}{L + \mu - 2d})}}.
\end{equation}
We can repeat the same process for $\hat{\rho}(\hat{K}_2)$, getting the following problem:
\begin{equation}\label{eq:prob-rect-main}
    \min_{a \in [\frac{L - \mu}{2}, \frac{L + \mu}{2}]} \frac{\frac{L + \mu }{2} - \sqrt{(\frac{L + \mu }{2})^2 + \frac{a^2 L^2}{a^2 - (\frac{L - \mu }{2})^2} - a^2}}{a - \frac{a L}{\sqrt{a^2 - (\frac{L - \mu }{2})^2}}}.
\end{equation}
Particularly, one can show that both \eqref{eq:prob-tri-main} and \eqref{eq:prob-rect-main} are approximately $1 - \Theta(\kappa^{-1.5})$ (see the supplement for details). Hence, we have $\hat{\rho}(\hat{K}) = 1 - \Theta(\kappa^{-1.5})$ by the sandwich argument. Together with the assumption that the vector field $F$ is continuously differentiable, we proved that negative momentum converges locally with this rate. This completes the proof.
\end{proof}
This shows that the optimal momentum parameter for minimax games is indeed negative and negative momentum with optimally tuned parameter does speed up the convergence of GDA locally, whose iteration complexity is $\bigo(\kappa^2)$~\citep{azizian2020tight}. However, the best existing lower bound on $\hat{K}$ is $\Omega(\kappa)$ iteration complexity~\citep{azizian2020accelerating, zhang2019lower}. Furthermore, the lower bound is tight as it is already achieved by EG and OGDA~\citep{mokhtari2020unified}. Thus we conclude that negative momentum is indeed a suboptimal algorithm.

\section{Related Works}

Polynomial-based iterative methods have long been used in solving linear systems. Two classical examples are the conjugate gradient method~\citep{hestenes1952methods} and the Chebyshev iteration~\citep{lanczos1952solution, golub1961chebyshev}, which forms the basis of some of the most used optimization methods such as Polyak momentum. For symmetric linear systems, \citet{fischer2011polynomial} provides a comprehensive study over the state of art on polynomial-based iterative methods. For non-symmetric linear systems, \citet{manteuffel1977tchebychev} discussed Chebyshev polynomial and showed that the iteration converges whenever the eigenvalues of the linear system lie in the open right half complex plane. Particularly, it was shown by~\citep{manteuffel1977tchebychev} that Chebyshev polynomial is optimal when the eigenvalues of the linear system lie within a complex ellipse, which inspires our work. For general non-symmetric linear systems, \citet{eiermann1983construction} used complex analysis tools to define, for a given compact set, its asymptotic convergence factor: it is the optimal asymptotic convergence rate a first-order method can achieve for all linear systems with spectrum in the set. Recently, \citet{azizian2020accelerating} used the tool of polynomial approximation to characterize acceleration in smooth games. \citet{pedregosa2020average} and \citet{scieur2020universal} used these ideas to develop methods that are optimal for the average-case.

\begin{figure*}[t]
	\centering
    \begin{subfigure}[t]{0.48\textwidth}
        \centering
        \includegraphics[width=0.90\textwidth]{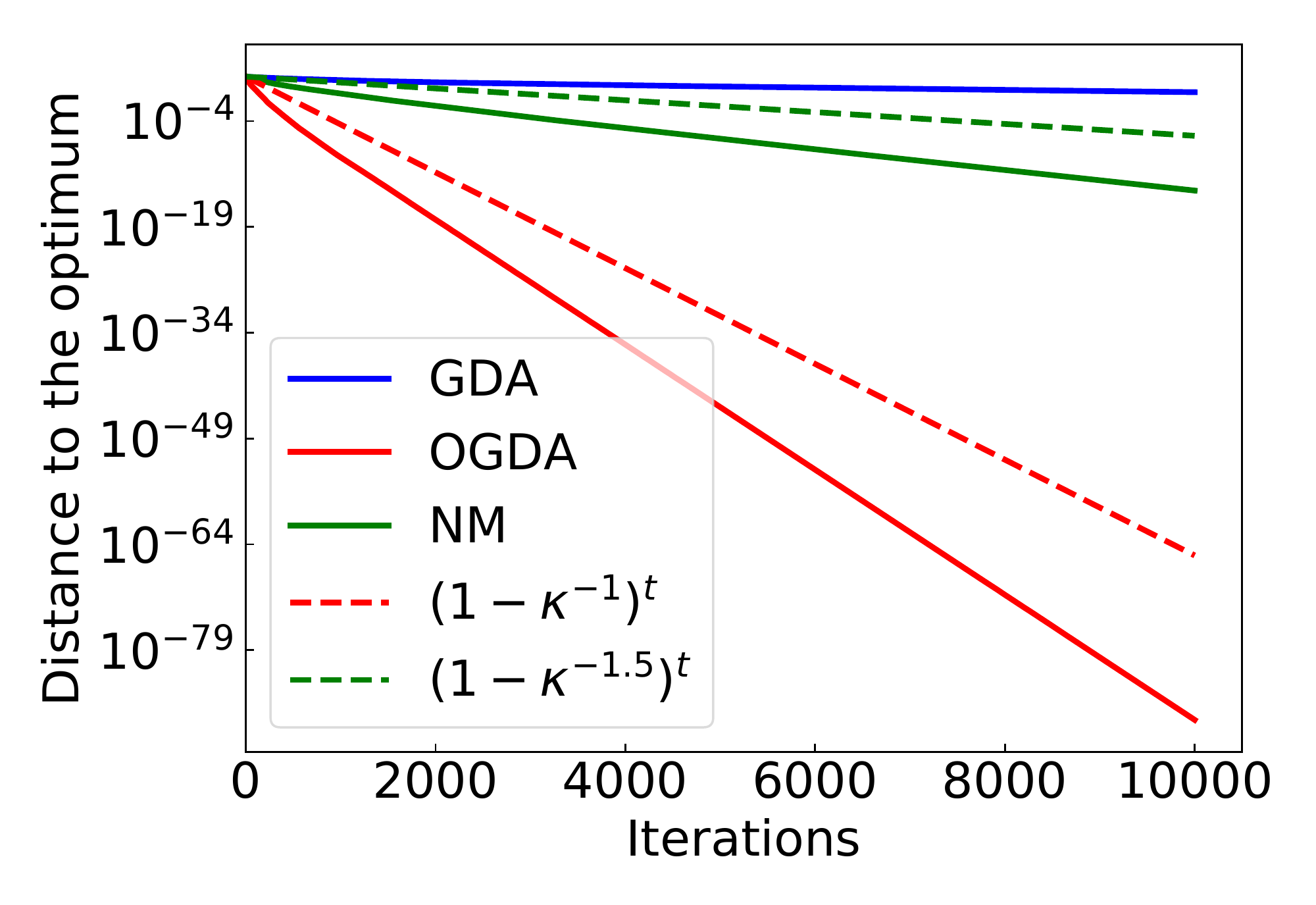}
        \vspace{-0.2cm}
    \end{subfigure}
    \begin{subfigure}[t]{0.48\textwidth}
        \centering
        \includegraphics[width=0.90\textwidth]{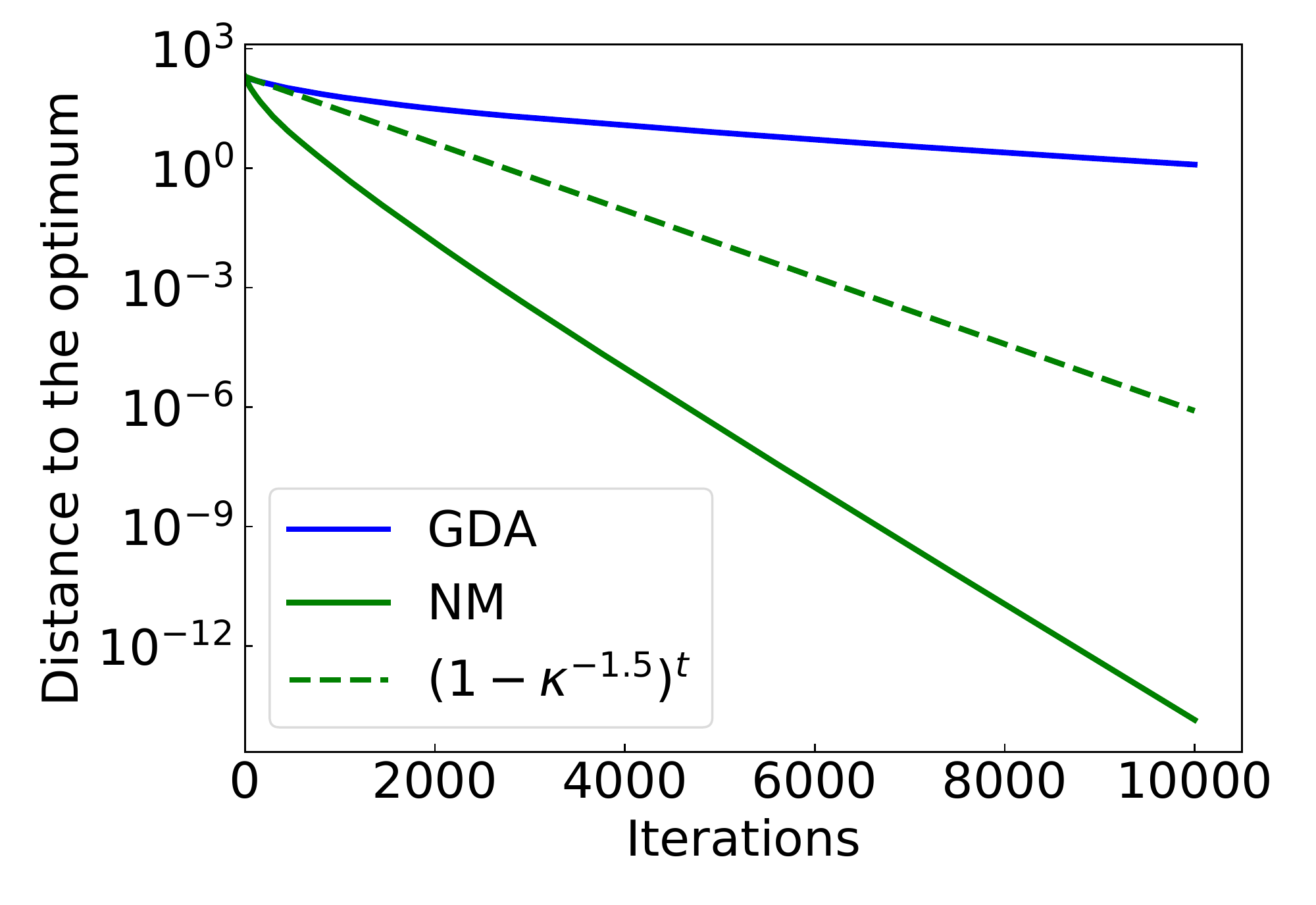}
        \vspace{-0.2cm}
    \end{subfigure}
    \vspace{-0.25cm}
	\caption{Distance to the optimum as a function of training iterations. Negative momentum accelerates GDA significantly on this quadratic minimax game. In particular, its convergence rate is slightly better than the worst-case rate of $1 - \kappa^{-1.5}$. However, negative momentum is outperformed by OGDA, whose convergence rate is approximately $1 - \kappa^{-1}$.}
	\label{fig:simulation}
\end{figure*}
In the context of minimax optimization, a line of recent work has studied various algorithms under different assumptions. For the strongly-convex strongly-concave case, \citet{tseng1995linear} and \citet{nesterov2006solving} proved that their algorithms find an $\epsilon$-saddle point with a gradient complexity of $\bigo(\kappa\ln(1/\epsilon))$ using a variational inequality approach. Using a different approach, \citet{gidel2018variational} and \citet{mokhtari2020unified} derived the same convergence results for OGDA. Particularly, \citet{mokhtari2020unified} unified the algorithm of OGDA and EG from the perspective of proximal point method, which gives sharp analysis. Notably, this convergence rate is known to be optimal to some extent~\citep{azizian2020accelerating}. Very recently, \citet{ibrahim2019linear, zhang2019lower} established fine-grained lower complexity bound among all the first-order algorithms in this setting, which was later achieved by the algorithms in~\citet{lin2020near, wang2020improved}. 
To our knowledge, the convergence rate of negative momentum has not been established in this setting before. The only known rate of negative momentum was for simple bilinear games~\citep{gidel2019negative}. Particularly, they showed that negative momentum with alternating updates achieves linear convergence, matching the rate of EG and OGDA. In this sense, we are the first to give an explicit rate of negative momentum for strongly-convex strongly-concave setting, though the rate is just local convergence rate.

More broadly, nonconvex-nonconcave problem has gained more attention due to its generality. However, there might be no \emph{Nash} (or even local Nash) equilibrium in that setting due to the loss of strong duality. To overcome that, different notations of equilibrium were introduced by taking into account the sequential structure of games~\citep{jin2019local, fiez2019convergence, zhang2020optimality, farnia2020gans}. In that setting, the main challenge is to find the right equilibrium and some algorithms~\citep{wang2019solving, adolphs2019local, mazumdar2019finding} have been proposed to achieve that.

\section{Numerical Simulations}\label{sec:simulation}
In this section, we compare the performance of negative momentum with Gradient-Descent-Ascent (GDA) and Optimistic Gradient-Descent-Ascent (OGDA) so as to verify our theoretical result on the convergence rate of negative momentum. In particular, we focus on the following quadratic minimax problem:
\begin{equation}
    \min_{\bx \in\real^d} \max_{\by \in \real^d} f(\bx, \by) = \frac{1}{2}\bx^\top \bA \bx + \bx^\top \bB \by - \frac{1}{2} \by^\top \bC \by
\end{equation}
where we set the dimension $d = 100$. 
The matrix $\bA$ and $\bC$ have eigenvalues $\left\{\frac{1}{ i}\right\}_{i=1}^{d}$, giving a condition number of $100$. For matrix $\bB$, we set it to be a random diagonal matrix with entries sampling from $[0, 1]$. For all algorithms, the iterates start with $\bx_0 = \mathbf{1}$ and $\by_0 = \mathbf{1}$. Figure~\ref{fig:simulation} shows that the distance to the optimum of negative momentum, GDA and OGDA versus the number of iterations for this quadratic minimax problem. For all methods, we tune their hyperparameters by grid-search. We can observe that all three methods converge linearly to the optimum. As expected, negative momentum performs better than GDA, but worse than OGDA. Moreover, both negative momentum and OGDA yield convergences rates that are slightly better than their worst-case rates.

\section{Discussion}
Although it may seem tempting to directly apply algorithmic techniques for minimization to minimax optimization, they can be provably suboptimal, as shown in this paper. The reason is that the dynamics of minimax optimization is different and considerably more complex. Thus we believe it is important to delve deeper and understand such game dynamics with multiple interacting objectives better.
Despite an existing line of work on accelerating GDA in smooth games, previously negative momentum was only analyzed for bilinear games. Due to the fact that negative momentum enjoys the same convergence rate as OGDA does in bilinear games, researchers are often confused with the difference between them and even call OGDA as ``negative momentum" (see~\citet{mokhtari2020unified} for example). Therefore, we believe our analysis of negative momentum is crucial as it highlights that negative momentum is fundamentally different from OGDA.

It is important to emphasize that we only provide local convergence rate of negative momentum in the paper. It is currently unknown whether negative momentum can attain the same geometric rate globally\footnote{It is now proved by~\citet{zhang2020unified} that negative momentum attains the same rate globally.}. We left this analysis for future work. In addition, it would be interesting to derive the optimal polynomial (hence optimal first-order algorithm) for smooth and strongly-monotone games. One promising way to achieve that is to finding the conformal mapping between the complement of $\hat{K}$ and the complement of unit disk, then Fabor polynomial~\citep{curtiss1971faber} can be adopted to derive the optimal polynomial.

\newpage
\section*{Acknowledgements}
We thank Shengyang Sun, Jenny Bao, Ricky Chen and Roger Grosse for detailed comments on early drafts. 
We thank all the anonymous reviewers (especially reviewer 2) for their careful reading of our manuscript and their many insightful comments and suggestions.
GZ would also like to thank for the support from Borealis AI fellowship. 
\bibliographystyle{plainnat}
\bibliography{reference.bib}

\newpage
\appendix
\onecolumn

\section{Proofs for Section~\ref{sec: poly-based-methods}}
\subsection{Proof of Theorem~\ref{thm: negative-momentum}}
Let $q_t(\lambda)$ be the rescaled and translated Chebyshev polynomial with degree $t$:
\begin{equation*}
    q_t(\lambda) \triangleq \frac{T_t(\frac{d - \lambda}{c})}{T_t(\frac{d}{c})}
\end{equation*}
We need to prove that for the region $K$ defined in \eqref{eq:ellipse}, $\mathcal{P}_t(\lambda; K) = q_t(\lambda)$. According to the definition of $\mathcal{P}_t(\lambda; K)$, we have
\begin{equation*}
    \mathcal{P}_t(\lambda; K) = \argmin_{p_t \in \Pi_t, \\ p_t(0)=1} \max_{\lambda \in K} |p_t(\lambda)|
\end{equation*}
Because $K$ is a bounded region, we know that the maximum modulus of an analytical function occurs on the boundary (according to maximum modulus principle). Let $B$ be the boundary of $K$ with the form
\begin{equation}\label{eq:boundary}
    B = \left\{\lambda \in \comp: \frac{(\realp\lambda - d)^2}{a^2} + \frac{(\imag\lambda)^2}{b^2} = 1 \right\}, \; d > a > 0, b > 0
\end{equation}
Instead of maximizing the modulus over the entire region $K$, we can take the maximum over the boundary $B$ and find the optimal polynomial $\mathcal{P}_t(\lambda;K)$ with
\begin{equation*}
    \mathcal{P}_t(\lambda; K) = \argmin_{p_t \in \Pi_t, \\ p_t(0)=1} \max_{\lambda \in B} |p_t(\lambda)|
\end{equation*}
With this reduction, we have the following:
\begin{lemma}
    Suppose $B$ does not include the origin in its interior, then we have
    \begin{equation*}
        \min_{\lambda \in B} |q_t(\lambda)| \leq \max_{\lambda \in B} |\mathcal{P}_t(\lambda; K)| \leq \max_{\lambda \in B} |q_t(\lambda)|
    \end{equation*}
\end{lemma}
\begin{proof}
    First, the second inequality holds by the definition of $\mathcal{P}_t(\lambda; K)$. We prove the first inequality by contradiction. Suppose that $\min_{\lambda \in B} |q_t(\lambda)| > \max_{\lambda \in B} |\mathcal{P}_t(\lambda; K)|$, then $|q_t(\lambda)| > |\mathcal{P}_t(\lambda; K)|$ for all $\lambda \in B$. By Rouch{\'e}'s Theorem~\citep{beardon2019complex}, we have the polynomial $q_t(\lambda) - \mathcal{P}_t(\lambda; K)$ has the same number of zeros in the interior of $B$ as $q_t(\lambda)$ does. Notice that $q_t(\lambda)$ has $t$ zeros inside $B$ and $q_t(0) - \mathcal{P}_t(0; K) = 0$. Because the origin $\lambda = 0$ is not in the interior of $B$, we thus conclude that $q_t(\lambda) - \mathcal{P}_t(\lambda; K)$ is a polynomial of degree $t$ with $t+1$ zeros, which is impossible. We therefore proved the first inequality.
\end{proof}
Given the sandwiching inequalities above, it suffices to show that
\begin{equation*}
    \lim_{t \rightarrow \infty} \left(\min_{\lambda \in B} |q_t(\lambda)| \right)^{1/t} = \lim_{t \rightarrow \infty} \left(\max_{\lambda \in B} |q_t(\lambda)| \right)^{1/t}.
\end{equation*}
According to the definition of Chebyshev polynomial $T_t$, we have
\begin{equation*}
    r(\lambda) \triangleq \lim_{t\rightarrow \infty }|q_t(\lambda)|^{1/t} = \left|e^{\cosh^{-1}\left(\frac{d-\lambda}{c}\right) - \cosh^{-1}\left(\frac{d}{c}\right)}\right|
\end{equation*}
One can easily show that $r(\lambda)$ is constant over the boundary $B$. Therefore, we prove that $\mathcal{P}_t(\lambda; K) = q_t(\lambda)$ as $t \rightarrow \infty$. We are now only left with the asymptotic convergence factor. We first consider the case of $a \neq b$, we know $\rho(K) = r(d + a)$ with the form
\begin{equation*}
    r(d + a) = \left|e^{\cosh^{-1}\left(\frac{a}{c}\right) - \cosh^{-1}\left(\frac{d}{c}\right)}\right| = \frac{d - \sqrt{d^2 + b^2 - a^2}}{a - b} 
\end{equation*}
For the special case of $a = b$, the ellipse is deformed into the circle, thus we have
\begin{equation*}
    \lim_{c \rightarrow 0 }r(d + a) = \lim_{c \rightarrow 0 }\left|e^{\cosh^{-1}\left(\frac{a}{c}\right) - \cosh^{-1}\left(\frac{d}{c}\right)}\right| = \lim_{c \rightarrow 0} \frac{a + \sqrt{a^2 - c^2}}{d + \sqrt{d^2 - c^2}} = a/d
\end{equation*}

\subsection{Proofs for Other Results}
\cornm*
\begin{proof}
As we showed in Theorem~\ref{thm: negative-momentum}, $\mathcal{P}_t(\lambda; K)$ is a rescaled and translated Chebyshev polynomial. Together with Lemma~\ref{lem:first-order}, we have
\begin{equation*}
    \bz_{t+1} - \bz^* = \frac{T_{t+1}(\frac{d - \bA}{c})}{T_{t+1}(\frac{d}{c})} (\bz_0 - \bz^*)
\end{equation*}
Using the recursion of Chebyshev polynomials (\ref{eq:cheby-recursion}), we have
\begin{equation*}
\begin{aligned}
    \bz_{t+1} - \bz^* &= 2\frac{T_t(\frac{d}{c})}{T_{t+1}(\frac{d}{c})} \frac{d - \bA}{c} (\bz_t - \bz^*) - \frac{T_{t-1}(\frac{d}{c})}{T_{t+1}(\frac{d}{c})} (\bz_{t-1} - \bz^*) \\
    &= \frac{2d}{c} \frac{T_t(\frac{d}{c})}{T_{t+1}(\frac{d}{c})} (\bz_t - \bz^*) - \frac{T_{t-1}(\frac{d}{c})}{T_{t+1}(\frac{d}{c})} (\bz_{t-1} - \bz^*) - \frac{2}{c} \frac{T_t(\frac{d}{c})}{T_{t+1}(\frac{d}{c})} F(\bz_t) \\
    &= \bz_{t} - \bz^* + \frac{T_{t-1}(\frac{d}{c})}{T_{t+1}(\frac{d}{c})} (\bz_t - \bz_{t-1}) - \frac{2}{c} \frac{T_t(\frac{d}{c})}{T_{t+1}(\frac{d}{c})} F(\bz_t) \\
\end{aligned}
\end{equation*}
Thus, we have
\begin{equation*}
    \bz_{t+1} = \bz_t - \eta_t F(\bz_t) + \beta_t (\bz_t - \bz_{t-1})
\end{equation*}
where $\eta_t = \frac{2}{c} \frac{T_t(\frac{d}{c})}{T_{t+1}(\frac{d}{c})}$ and $\beta_t = \frac{T_{t-1}(\frac{d}{c})}{T_{t+1}(\frac{d}{c})}$.
Again appealing to recursion (\ref{eq:cheby-recursion}), we can generate $\eta_t$ and $\beta_t$ recursively:
\begin{equation*}
    \eta_t = \left[d - (c/2)^2 \eta_{t-1} \right]^{-1}, \quad \beta_t = d\eta_t - 1
\end{equation*}
With such recursion, one can easily get the fix points of $\eta$ and $\beta$:
\begin{equation}\label{eq:asymptotic-params}
    \eta = 2\frac{d - \sqrt{d^2 - c^2}}{c^2}, \quad \beta = d \eta - 1
\end{equation}
We now proceed to prove that Polyak momentum with fixed $\eta, \beta$ can achieve the convergence rate. We first introduce the concept of $\rho$-convergence region for momentum method:
\begin{equation*}
    S(\eta, \beta, \rho) = \left\{\lambda \in \comp: \forall x \in \comp, x^2 - (1 - \eta \lambda + \beta)z + \beta \leq \rho \Rightarrow |x| \leq \rho \right\}
\end{equation*}
We call it the $\rho$-convergence region of the momentum method as it corresponds to the maximal regions of the complex plane where the momentum method converges at rate $\rho$.
It has been shown by~\citet{niethammer1983analysis} that $S(\eta, \beta, \rho)$ is a complex ellipse on complex plane.
\begin{lemma}[{\citet[Cor.~6]{niethammer1983analysis}}]\label{lem:rho-regime}
    For $\beta \leq \rho$ and $\rho > 0$, we have
    \begin{equation*}
        S(\eta, \beta, \rho) = \left\{\lambda \in \comp: \frac{(1 - \eta \realp \lambda + \beta)^2}{(1 + \tau)^2} + \frac{(\eta \imag \lambda)^2}{(1 - \tau)^2} \leq \rho^2   \right\}
    \end{equation*}
    where $\tau = \beta / \rho^2$.
\end{lemma}
Taking the values of $\eta, \beta$ by (\ref{eq:asymptotic-params}) and $\rho = \frac{d - \sqrt{d^2 + b^2 - a^2}}{a - b}$, we have
$S(\eta, \beta, \rho)$ is the same as $K$ in~\eqref{eq:ellipse}.
Therefore, we can achieve the asymptotic convergence rate even with constant $\eta$ and $\beta$.
\end{proof}

\corgda*
\begin{proof}
As we shown in Theorem~\ref{thm: negative-momentum}, the asymptotic optimal polynomial is rescaled and translated polynomial
\begin{equation*}
    \mathcal{P}_t(\lambda; K) = \frac{T_t(\frac{d - \lambda}{c})}{T_t(\frac{d}{c})}
\end{equation*}
In the case of $\alpha = \beta = 0$, we have
\begin{equation*}
    \lim_{c \rightarrow 0} \mathcal{P}_t(\lambda; K) = \lim_{c \rightarrow 0} \frac{T_t(\frac{d - \lambda}{c})}{T_t(\frac{d}{c})} = \frac{(d - \lambda)^t + (d - \lambda)^{-t}}{d^t + d^{-t}}
\end{equation*}
For large $t$, we get $\mathcal{P}_t(\lambda; K) = (1 - \lambda / d)^t$.
\end{proof}

\cormodulus*
\begin{proof}
We can show that the modulus of $\mathcal{P}_t(\lambda; K)$ has the form:
\begin{equation*}
    |\mathcal{P}_t(\lambda; K)| = \left|e^{\cosh^{-1}\left(\frac{d-\lambda}{c}\right) - \cosh^{-1}\left(\frac{d}{c}\right)}\right| 
    =  e^{\mathrm{Re}\left(\cosh^{-1}\left(\frac{d-\lambda}{c}\right) - \cosh^{-1}\left(\frac{d}{c}\right)\right)}
\end{equation*}
One can show that $\frac{d - \lambda}{c}$ maps points of boundary (\ref{eq:boundary}) to the family of ellipse with the center $0$ and foci at $-1$ and $1$. According to the property of $\cosh^{-1}$, it maps such family of ellipse to the vertical line $x = \mathrm{const}$. We thus conclude that the modulus of $\mathcal{P}_t(\lambda; K)$ is constant on the boundary.
\end{proof}

\section{Proofs for Section~\ref{sec:main-result}} 
\subsection{Proof of Theorem~\ref{thm:main-result}}
Let's first prove the result for $\hat{\rho}(\hat{K}_1)$. By Lemma~\ref{lem:alternative-thm}, we have
\begin{equation*}
    r(L; d_1^*, {c_1^2}^*) = r(\mu+\sqrt{L^2 - \mu^2}i; d_1^*, {c_1^2}^*),
\end{equation*}
which implies that both $L$ and $\mu + \sqrt{L^2 - \mu^2}i$ are on the boundary of the same complex ellipse with the center $d_1^*$ and foci at $d_1^* - c_1^*$ and $d_1^* + c_1^*$ according to Lemma~\ref{lem:convergence-factor}. %
Then by Theorem~\ref{thm: negative-momentum} and Corollary~\ref{cor:max-modulus}, we can reduce the computation of $\hat{\rho}(\hat{K}_1)$ to the following constrained problem:
\begin{equation}\label{eq:constrainted-prob}
\begin{aligned}
    &\hat{\rho}(\hat{K}_1) := \min_{a, b, d} \tfrac{d - \sqrt{d^2 + b^2 - a^2}}{a - b}, \\ &\text{s.t. } E_{a,b,d}(L) = E_{a, b, d}(\mu + \sqrt{L^2 - \mu^2}i) = 1
\end{aligned}
\end{equation}
which involves three free variables and two constraints. By two constraints, we have
\begin{equation*}
    b^2 = \frac{(L + \mu)(L - d)^2}{L + \mu - 2d} > (L - d)^2 = a^2.
\end{equation*}
Therefore, the optimal momentum $\beta$ for $\hat{K}_1$ is negative. For $\hat{K}_2$, we follow the same procedure and have
\begin{equation*}
    b^2 = \frac{(L^2 - \mu^2) a^2}{a^2 - (\frac{L - \mu}{2})^2}, \; a \in [\frac{L - \mu}{2}, \frac{L + \mu}{2}].
\end{equation*}
In the case of $L^2 > \mu^2 + \mu L$, we have $b^2 > a^2$ and therefore the optimal momentum is also negative. Hence, we conclude that the optimal momentum for $\hat{K}$ is negative.

Next, one can further simplify the problem (\ref{eq:constrainted-prob}) to a single variable minimization task:
\begin{equation}\label{eq:prob-tri}
    \min_{d \in [\frac{L}{2}, \frac{\mu+L}{2}]} \frac{d - \sqrt{\frac{2d(L - d)^2}{L + \mu - 2d} + d^2}}{(L - d)(1 - \sqrt{\frac{L + \mu}{L + \mu - 2d})}}.
\end{equation}
We can repeat the same process for $\hat{\rho}(\hat{K}_2)$, getting the following problem:
\begin{equation}\label{eq:prob-rect}
    \min_{a \in [\frac{L - \mu}{2}, \frac{L + \mu}{2}]} \frac{\frac{L + \mu }{2} - \sqrt{(\frac{L + \mu }{2})^2 + \frac{a^2 L^2}{a^2 - (\frac{L - \mu }{2})^2} - a^2}}{a - \frac{a L}{\sqrt{a^2 - (\frac{L - \mu }{2})^2}}}.
\end{equation}
Let us first focus on (\ref{eq:prob-tri}). Let $\kappa:=L/\mu$, $x:=d/\mu-\frac{\kappa}{2}$, $D:=d/\mu$, then
\begin{equation*}
\begin{aligned}
(\ref{eq:prob-tri}) &= \min_{x\in[0,1/2]}\frac{-D\sqrt{\kappa+1-2D}+\sqrt{2D(\kappa-D)^2+D^2(\kappa+1-2D)}}{(\kappa-D)\left(\sqrt{\kappa+1}-\sqrt{\kappa+1-2D}\right)}\\
&= \min_{x\in[0,1/2]}\frac{\left(-D\sqrt{1-2x}+\sqrt{D^2+(2\kappa-3D)D\kappa}\right)\cdot\left(\sqrt{\kappa+1}+\sqrt{1-2x}\right)}{2D\cdot (\kappa-D)}.
\end{aligned}
\end{equation*}
Since $\forall x\ge 0$, $\sqrt{1+x}\ge 1+\frac{x}{2}-\frac{x^2}{8}$ and for any $0\le x\le \frac{1}{2}$, $|x-3x^2|\le x$,
\begin{equation*}
\begin{aligned}
\sqrt{D^2+(2\kappa-3D)D\kappa}&= \sqrt{\frac{1}{4}\kappa^3 +\frac{1}{4}\kappa^2-\kappa^2x-3\kappa x^2+\kappa x + x^2}\\
&=\frac{1}{2}\kappa^{1.5}\cdot\sqrt{1+\frac{1-4x}{\kappa}+\frac{4x-12x^2}{\kappa^2}+\frac{4x^2}{\kappa^3}}\\
&\ge \frac{1}{2}\kappa^{1.5}\cdot \left[1+\frac{1-4x}{2\kappa}+\frac{2x-6x^2}{\kappa^2}+\frac{2x}{\kappa^3}-\frac{1}{8}\left(\frac{1-4x}{\kappa}+\frac{4x-12x^2}{\kappa^2}+\frac{4x}{\kappa^3}\right)^2\right]\\
&\ge \frac{1}{2}\kappa^{1.5}\cdot\left(1+\frac{1-4x}{2\kappa}-\frac{8}{\kappa^2}\right).
\end{aligned}
\end{equation*}
Meanwhile
\begin{equation*}
\begin{aligned}
\sqrt{\kappa+1}+\sqrt{1-2x}\ge \sqrt{\kappa}+\frac{1}{2\sqrt{\kappa}}+\sqrt{1-2x}-\frac{1}{8\kappa^{1.5}}.
\end{aligned}
\end{equation*}
Thus
\begin{equation*}
\begin{aligned}
&\left(-D\sqrt{1-2x}+\sqrt{D^2+(2\kappa-3D)D\kappa}\right)\cdot\left(\sqrt{\kappa+1}+\sqrt{1-2x}\right)\\
\ge & \left(-\frac{\kappa}{2}\sqrt{1-2x}-x\sqrt{1-2x}+\frac{1}{2}\kappa^{1.5}\left(1+\frac{1-4x}{2\kappa}-\frac{8}{\kappa^2}\right)\right)\cdot \left(\sqrt{\kappa}+\frac{1}{2\sqrt{\kappa}}+\sqrt{1-2x}-\frac{1}{8\kappa^{1.5}}\right)\\
\ge & \frac{1}{2}\left(\kappa^2-2x\sqrt{1-2x}\sqrt{\kappa}-\frac{\sqrt{1-2x}}{2}\sqrt{\kappa}+\frac{1-4x}{2}\sqrt{1-2x}\sqrt{\kappa}-24\right)\\
\ge & \frac{1}{2}\left(\kappa^2-2\sqrt{\kappa}-24\right).
\end{aligned}
\end{equation*}
Therefore
\begin{equation*}
\begin{aligned}
(\ref{eq:prob-tri})&\ge \frac{\frac{1}{2}\left(\kappa^2-2\sqrt{\kappa}-24\right)}{\frac{\kappa^2}{2}-2x^2}\ge \frac{\kappa^2-2\sqrt{\kappa}-24}{\kappa^2-1}\\
&\ge 1-\frac{2}{\kappa^{1.5}}-\frac{24}{\kappa^2}.
\end{aligned}
\end{equation*}
Meanwhile, if we choose $x=\frac{1}{4}$, we can show that
\begin{equation*}
\begin{aligned}
(\ref{eq:prob-tri}) &\le  \frac{\left(-D\sqrt{1-2x}+\sqrt{D^2+(2\kappa-3D)D\kappa}\right)\cdot\left(\sqrt{\kappa+1}+\sqrt{1-2x}\right)}{2D\cdot (\kappa-D)}\\
&\le \frac{\left[-(\frac{\kappa}{2}+\frac{1}{4})\sqrt{\frac{1}{2}}+\sqrt{(\frac{\kappa}{2}+\frac{1}{4})^2+\kappa(\frac{\kappa}{2}+\frac{1}{4})(\frac{\kappa}{2}-\frac{3}{4})}\right]\cdot\left(\sqrt{\kappa+1}+\sqrt{\frac{1}{2}}\right)}{(\kappa+\frac{1}{2})(\frac{\kappa}{2}-\frac{1}{4})}\\
&= \frac{\left[-\sqrt{2}(\frac{\kappa}{2}+\frac{1}{4}) + \sqrt{\kappa^3+\frac{1}{4}\kappa+\frac{1}{4}}\right]\cdot\left(\sqrt{\kappa+1}+\sqrt{\frac{1}{2}}\right)}{\kappa^2-\frac{1}{4}}\\
&\le \frac{\left(-\frac{\kappa}{\sqrt{2}}-\frac{\sqrt{2}}{4}+\kappa^{1.5}+\frac{1}{2\sqrt{\kappa}}+\frac{1}{2\kappa^{1.5}}\right)\cdot\left(\sqrt{\kappa}+\frac{1}{2\sqrt{\kappa}}+\frac{1}{\sqrt{2}}\right)}{\kappa^2-\frac{1}{4}}\\
&\le \frac{\kappa^2-\frac{\sqrt{2}}{2}\sqrt{\kappa}+2}{\kappa^2-\frac{1}{4}} \le 1-\frac{\sqrt{2}}{2}\kappa^{-1.5}+\frac{9}{4}\kappa^{-2}.
\end{aligned}
\end{equation*}
Therefore
\begin{equation*}
    \min_{d \in [\frac{L}{2}, \frac{\mu+L}{2}]} \frac{d - \sqrt{\frac{2d(L - d)^2}{L + \mu - 2d} + d^2}}{(L - d)(1 - \sqrt{\frac{L + \mu}{L + \mu - 2d})}} = 1-\Theta(\kappa^{-1.5}).
\end{equation*}
Now let us focus on (\ref{eq:prob-rect}). Let $\kappa:=L/\mu$, $x:=a/\mu$.
\begin{equation*}
\begin{aligned}
(\ref{eq:prob-rect}) &\ge \min_{x\in[0,1]}\frac{-\frac{\kappa+1}{2}+\sqrt{\frac{x^2\kappa^2}{x^2-(\frac{\kappa-1}{2})^2}}}{-x + \frac{x \kappa}{\sqrt{x^2-(\frac{\kappa-1}{2})^2}}}= \min_{x\in[0,1]}\frac{x\kappa - \frac{\kappa+1}{2}\sqrt{x^2-\left(\frac{\kappa-1}{2}\right)^2}}{x\kappa - x\sqrt{x^2-\left(\frac{\kappa-1}{2}\right)^2}}\\
&\ge \min_{x\in[0,1]}\frac{1-\frac{\kappa+1}{2x\kappa}\sqrt{x^2-\left(\frac{\kappa-1}{2}\right)^2}}{1-\frac{1}{\kappa}\sqrt{x^2-\left(\frac{\kappa-1}{2}\right)^2}}.
\end{aligned}
\end{equation*}
Since $\frac{\kappa}{2}\le x\le \frac{\kappa+1}{2}$,
\begin{equation*}
\frac{\kappa+1}{2x\kappa}-\frac{1}{\kappa}\le \frac{\kappa+1}{\kappa^2}-\frac{1}{\kappa}=\frac{1}{\kappa^2}.
\end{equation*}
Thus
\begin{equation*}
\begin{aligned}
(\ref{eq:prob-rect}) &\ge \min_{x\in[0,1]}\left\{1 - \frac{\frac{1}{\kappa^2}\sqrt{x^2-\left(\frac{\kappa-1}{2}\right)^2}}{1-\frac{1}{\kappa}\sqrt{x^2-\left(\frac{\kappa-1}{2}\right)^2}}\right\}.
\end{aligned}
\end{equation*}
Assume that $\kappa\ge 4$. Then $$\frac{1}{\kappa}\sqrt{x^2-\left(\frac{\kappa-1}{2}\right)^2}\le\frac{1}{\kappa}\sqrt{\left(\frac{\kappa+1}{2}\right)^2-\left(\frac{\kappa-1}{2}\right)^2}=\frac{1}{\sqrt{\kappa}}\le \frac{1}{2},$$
and therefore
$$(\ref{eq:prob-rect})\ge 1-\frac{\frac{1}{\kappa\sqrt{\kappa}}}{1/2}=1-\frac{2}{\kappa^{1.5}}.$$
Meanwhile, if we choose $x = \frac{\kappa}{2}$, then 
\begin{equation*}
\begin{aligned}
(\ref{eq:prob-rect}) &\le \frac{-\frac{\kappa+1}{2}+\sqrt{(\frac{\kappa+1}{2})^2+\frac{\kappa^4}{2\kappa-1}-\frac{\kappa^2}{4}}}{-\frac{\kappa}{2}+\kappa^2\cdot\frac{1}{\sqrt{2\kappa-1}}}\\
&\le \frac{\sqrt{\kappa^4+\kappa^2-\frac{1}{4}}-\frac{\kappa+1}{2}\sqrt{2\kappa-1}}{\kappa^2-\frac{\kappa}{2}\sqrt{2\kappa-1}}\\
&\le \frac{\kappa^2+\frac{1}{2}-\frac{\kappa+1}{2}\sqrt{2\kappa-1}}{\kappa^2-\frac{\kappa}{2}\sqrt{2\kappa-1}}\\
&\le 1-\frac{\sqrt{2\kappa-1}-1}{2\kappa^2}.
\end{aligned}
\end{equation*}
In other words, we can show that
\begin{equation*}
    \hat\rho(\hat K_1)=1-\Theta(\kappa^{-1.5}), \quad \hat\rho(\hat K_2)=1-\Theta(\kappa^{-1.5}).
\end{equation*}

Recall that $\hat{\rho}(\hat{K})$ is the spectral radius of the augmented Jacobian $\rho(\bar{\jacobian})$ in~\eqref{eq:aug-nm}. In order to prove the local convergence claim in the theorem, we will use the following lemma to connect $\hat{\rho}(\hat{K})$ with the local asymptotic rate. This lemma follows by linearizing the vector field around $\bz^*$. Similar results can be found in many texts, see for instance,~Theorem 2.12~in~\citet{olver2015nonlinear}.
\begin{prop}[Local convergence rate from Jacobian eigenvalue]
\label{prop:local}
For a discrete dynamical system $\bz_{t+1} = G(\bz_t) = \bz_t - F(\bz_t)$, if the spectral radius $\rho(\jacobian_G(\bz^*))=1-\Delta<1$, then there exists a neighborhood $U$ of $\bz^*$ such that for any $\bz_0\in U$,
\begin{equation*}
\Vert \bz_t-\bz^*\Vert_2 \le C\left(1-\frac{\Delta}{2}\right)^t\Vert \bz_0-\bz^*\Vert_2,
\end{equation*}
where $C$ is some constant.
\end{prop}
\begin{proof}
By Lemma 5.6.10~\citep{horn2012matrix}, since $\rho(\jacobian_G(\bz^*))=1-\Delta$, there exists a matrix norm $\Vert\cdot \Vert$ induced by vector norm $\Vert\cdot\Vert$ such that $\Vert \jacobian_G(\bz^*)\Vert < 1-\frac{3\Delta}{4}$. Now consider the Taylor expansion of $G(\bz)$ at the fixed point $\bz^*$:
\begin{equation*}
    G(\bz)=G(\bz^*)+\jacobian_G(\bz^*)(\bz-\bz^*)+R(\bz-\bz^*),
\end{equation*}
where the remainder term satisfies
\begin{equation*}
    \lim_{\bz\to\bz^*}\frac{R(\bz-\bz^*)}{\Vert \bz-\bz^*\Vert}=0.
\end{equation*}
Therefore, we can choose $0<\delta$ such that whenever $\Vert \bz-\bz^*\Vert<\delta$, $\Vert R(\bz-\bz^*)\Vert\le \frac{\Delta}{4}\Vert \bz-\bz^*\Vert.$ In this case,
\begin{equation*}
\begin{aligned}
    \Vert G(\bz)-G(\bz^*)\Vert &\le \Vert\jacobian_G(\bz^*)(\bz-\bz^*)\Vert+\Vert R(\bz-\bz^*)\Vert\\
    &\le \Vert\jacobian_G(\bz^*)\Vert \Vert \bz-\bz^* \Vert + \frac{\Delta}{4}\Vert \bz-\bz^*\Vert\\
    &\le \left(1-\frac{\Delta}{2}\right)\Vert \bz-\bz^*\Vert.
\end{aligned}
\end{equation*}
In other words, when $\bz_0\in U=\left\{\bz | \; \Vert \bz-\bz^*\Vert < \delta\right\}$,
\begin{equation*}
    \Vert \bz_t-\bz^*\Vert \le \left(1-\frac{\Delta}{2}\right)^t\Vert \bz_0-\bz^*\Vert.
\end{equation*}
By the equivalence of finite dimensional norms, there exists constants $c_1,c_2>0$ such that 
\begin{equation*}
    \forall \bz,\quad  c_1\Vert \bz\Vert_2\le \Vert \bz\Vert \le  c_2\Vert \bz\Vert_2.
\end{equation*}
Therefore
\begin{equation*}
    \Vert \bz_t-\bz^*\Vert_2 \le \frac{c_2}{c_1}\left(1-\frac{\Delta}{2}\right)^t\Vert \bz_0-\bz^*\Vert_2.
\end{equation*}
\end{proof}

\subsection{Proofs for Other Results}
\lemreduction*
\begin{proof}
    According to Lemma~\ref{lem:rho-regime}, we have the $\rho$-convergence region for momentum method with $\eta, \beta$:
    \begin{equation*}
        S(\eta, \beta, \rho) = \left\{\lambda \in \comp: \frac{(1 - \eta \realp \lambda + \beta)^2}{(1 + \beta/\rho^2)^2} + \frac{(\eta \imag \lambda)^2}{(1 - \beta/\rho^2)^2} \leq \rho^2   \right\}
    \end{equation*}
    The convergence rate for any particular choice of $\eta, \beta$ is the smallest $\rho$ such that $S(\eta, \beta, \rho)$ tightly covers the region $K$. By setting $d = \frac{1 + \beta}{\eta}$ and $c^2 = \frac{4 \beta}{\eta^2}$, one can show that Chebyshev iteration have the same rate on $S(\eta, \beta, \rho)$ as we can transform $S(\eta, \beta, \rho)$ to
    \begin{equation*}
        \left\{\lambda \in \comp: \frac{(\realp \lambda - d)^2}{(\frac{\rho}{\eta} + \frac{\beta}{\eta \rho})^2} + \frac{(\imag \lambda)^2}{(\frac{\rho}{\eta} - \frac{\beta}{\eta \rho})^2} \leq 1   \right\}
    \end{equation*}
    with $c^2 = (\frac{\rho}{\eta} + \frac{\beta}{\eta \rho})^2 - (\frac{\rho}{\eta} - \frac{\beta}{\eta \rho})^2 = \frac{4\beta}{\eta^2}$. On the other side for any Chebyshev iteration with parameters $d, c^2 \in \real$, we can take $\eta = 2\frac{d - \sqrt{d^2 - c^2}}{c^2}$ and $\beta = d \eta - 1$.
\end{proof}

\lemalter*
\begin{proof}
To prove this Lemma, we first recall the Alternative theorem from functional analysis (see e.g.~\citet{bartle1964elements}):
\begin{theorem}[Alternative theorem]
    If $\{f_i(x, y) \}$ is a finite set of real valued functions of two real variables, each of which is continuous on a closed and bounded region $S$ and we define
    \begin{equation*}
        m(x, y) = \max_i f_i(x, y)
    \end{equation*}
    then $m(x, y)$ takes on a minimum at some point $(x^*, y^*)$ in the region $S$. If $(x^*, y^*)$ is in the interior of $S$, then one of the following hold:
    \begin{enumerate}
        \item The point $(x^*, y^*)$ is a local minimum of $f_i(x, y)$ for some $i$ such that $m(x^*, y^*) = f_i (x^*, y^*)$.
        \item The point $(x^*, y^*)$ is a local minimum of among the locus $\{(x, y) \in S| f_i (x, y) = f_j(x, y) \}$ for some $i$ and $j$ such that $m(x^*, y^*) = f_i (x^*, y^*) = f_j(x^*, y^*)$.
        \item The point $(x^*, y^*)$ is such that for some $i, j$ and $k$ such that $m(x^*, y^*) = f_i (x^*, y^*) = f_j(x^*, y^*) = f_k(x^*, y^*)$.
    \end{enumerate}
\end{theorem}

Here we take the triangle $H_1$ as the example. Recall we have the min-max problem:
\begin{equation*}
    \min_{d, c^2 \in \real} \max \left\{r(L; d, c^2), r(\mu + \sqrt{L^2 - \mu^2}i; d, c^2)  \right\}
\end{equation*}
It is obvious that the solution to the min-max problem lies in the open region of $\real$ and there is some compact set $S \subset \real$ which contains the solution in its interior. Therefore, we can apply the Alternative theorem.
It is easily shown that 
\begin{equation*}
\begin{aligned}
    r(L; L, 0) &< r(\mu + \sqrt{L^2 - \mu^2}i; L, 0) \\
    r(L; \mu, \mu^2 - L^2) &> r(\mu + \sqrt{L^2 - \mu^2}i; \mu, \mu^2 - L^2) \\
\end{aligned}
\end{equation*}
Since there is only one local minimum on each surface, the Alternative theorem yields that the solution must occur along the intersection of the two surfaces. Therefore, we finish the proof.
\end{proof}

\end{document}